\documentclass[reqno,11pt]{amsart}
\usepackage{amssymb}

\usepackage{graphicx}
\usepackage{pstricks}
\usepackage{amsmath}
\usepackage{amsmath}
\usepackage{amsxtra}
\usepackage{graphicx}
\usepackage{pstricks}
\usepackage{amsmath}
\usepackage{amsmath}
\usepackage{amsxtra}

\theoremstyle{plain}
\newtheorem{theorem}{Theorem}[section]
\newtheorem{lemma}[theorem]{Lemma}

\newtheorem{definition}[theorem]{Definition}

\newtheorem{problem}[theorem]{Problem}

\theoremstyle{definition}
\newtheorem{notation}[theorem]{Notation}
\newtheorem{example}[theorem]{Example}
\newtheorem{remark}[theorem]{Remark}

\numberwithin{equation}{section}
\begin{document}
\title[2-variable subnormal completion problem]{A new approach to the $2$%
-variable \\
subnormal completion problem}
\author{Ra\'{u}l E. Curto}
\address[Ra\'{u}l E. Curto]{Department of Mathematics \\
The University of Iowa\\
Iowa City, Iowa 52242}
\email[Ra\'{u}l E. Curto]{rcurto@math.uiowa.edu}
\urladdr{http://www.math.uiowa.edu/\symbol{126}rcurto/}
\thanks{The first named author was partially supported by NSF Grants
DMS-0400741 and DMS-0801168.}

\author{Sang Hoon Lee}
\address[Sang Hoon Lee]{Department of Mathematics\\
Chungnam National University\\
Daejeon, 305-764, Korea }
\email[Sang Hoon Lee]{shlee@math.cnu.ac.kr}
\urladdr{}
\thanks{The second named author was partially supported by a National 
Research Foundation of Korea Grant funded by the Korean 
Government (2009-0085279).}

\author{Jasang Yoon}
\address[Jasang Yoon]{Department of Mathematics\\
The University of Texas-Pan American\\
Edinburg, Texas 78539}
\email[Jasang Yoon]{yoonj@utpa.edu}
\urladdr{http://www.math.utpa.edu/\symbol{126}yoonj/}
\thanks{The third named author was partially supported by a Faculty Research
Council Grant at The University of Texas-Pan American.}
\dedicatory{}
\thanks{}

\begin{abstract}
We study the Subnormal Completion Problem (SCP) for $2$-variable weighted
shifts. \ We use tools and techniques from the theory of truncated moment
problems to give a general strategy to solve SCP. \ We then show that when
all quadratic moments are known (equivalently, when the initial segment of
weights consists of five independent data points), the natural necessary
conditions for the existence of a subnormal completion are also sufficient.
\ To calculate explicitly the associated Berger measure, we compute the
algebraic variety of the associated truncated moment problem; it turns out
that this algebraic variety is precisely the support of the Berger measure
of the subnormal completion.
\end{abstract}

\maketitle

\section{Introduction}

We present a new approach to the Subnormal Completion Problem (SCP) for $2$%
-variable weighted shifts. \ It employs the localizing matrices introduced
and studied in \cite{tcmp4} in the context of the truncated $K$-moment
problem ($K$-TMP). \ This helps identify potential candidates for weights,
and makes the problem more accessible.

We first give a general strategy to solve SCP, and we later apply it to
solve the SCP with \textit{quadratic data}. \ That is, given an initial set
of weights $\Omega _{1}$ consisting of five independent data points ($\alpha
_{00}$, $\beta _{00}$, $\alpha _{10}$, $\alpha _{01}$ and $\beta _{01}$), we
prove that the natural necessary condition for the existence of a subnormal
completion is also sufficient. \ Concretely, associated to the five given
weights is a $3\times 3$ \textit{moment matrix} $M(\Omega _{1})$, whose
positive semi-definiteness is a necessary condition for the existence of a
subnormal completion; in symbols, $M(\Omega _{1}):=(\gamma _{\mathbf{u}+%
\mathbf{v}})_{\mathbf{u},\mathbf{v}\in \mathbb{Z}_{+}^{2},\left| \mathbf{u}%
\right| ,\left| \mathbf{v}\right| \leq 1}$, where $\gamma _{00}:=1$, $\gamma
_{10}:=\alpha _{00}^{2}$, $\gamma _{01}:=\beta _{00}^{2}$, $\gamma
_{20}:=\alpha _{10}^{2}\alpha _{00}^{2}$, $\gamma _{11}:=\alpha
_{01}^{2}\beta _{00}^{2}$, and $\gamma _{02}:=\beta _{01}^{2}\beta _{00}^{2}$%
. \ We prove that the necessary condition $M(\Omega _{1})\geq 0$ turns out
to be sufficient for the existence of a representing measure $\mu $
supported in $\mathbb{R}_{+}^{2}$ and satisfying the property $\operatorname{supp}%
\;\mu \cap (0,+\infty )^{2}\neq \varnothing $; the measure $\mu $ then gives
rise to a subnormal completion of $\Omega _{1}$. \ Once we know that a
representing measure exists, we use techniques from the theory of truncated
moment problems to find a concrete expression for it. \ 

As a first step, we build \textit{new} weights $\alpha _{20}$, $\alpha _{11}$%
, $\alpha _{02}$ and $\beta _{02}$, and we use them to construct the \textit{%
localizing matrices} $M_{x}(\hat{\Omega}_{3})$ and $M_{y}(\hat{\Omega}_{3})$%
, where $\hat{\Omega}_{3}$ is a proposed extension of $\Omega _{1}$. \ The
positive semi-definiteness of $M(\Omega _{1})$ is then used to establish
that the localizing matrices $M_{x}(\hat{\Omega}_{3})$ and $M_{y}(\hat{\Omega%
}_{3})$ can be made positive semi-definite for suitable choices of the new
weights $\alpha _{20}$, $\alpha _{11}$, $\alpha _{02}$ and $\beta _{02}$. \
That is, the condition $M(\Omega _{1})\geq 0$ triggers the two conditions $%
M_{x}(\hat{\Omega}_{3})\geq 0$ and $M_{y}(\hat{\Omega}_{3})\geq 0$ for
appropriate values of $\alpha _{20}$, $\alpha _{11}$, $\alpha _{02}$ and $%
\beta _{02}$. \ Once that happens, we prove that a \textit{flat} (i.e.,
rank-preserving) extension $M(\hat{\Omega}_{3})$ of $M(\Omega _{1})$ exists,
thereby giving rise to a unique representing measure $\mu $ for $\hat{\Omega}%
_{3}$, which is the Berger measure of the subnormal completion. \ The
explicit form of $\mu $ can be obtained by first determining the support of $%
\mu $, which agrees with the algebraic variety of $\hat{\Omega}_{3}$.

In one variable, SCP was stated and solved in \cite{RGWSI}: \ 

\begin{problem}
(One-Variable Subnormal Completion Problem) \ Given $m\geq 0$ and a finite
collection of positive numbers $\Omega _{m}\equiv \{\alpha _{k}\}_{k=0}^{m}$%
, find necessary and sufficient conditions on $\Omega _{m}$ to guarantee the
existence of a subnormal weighted shift whose initial weights are given by $%
\Omega _{m}$.
\end{problem}

Since subnormality implies hyponormality, the condition $\alpha _{0}\leq
\alpha _{1}\leq \alpha _{2}\leq \cdots \leq \alpha _{m}$ is obviously
necessary; moreover, it is easy to dispose of the case when $\alpha
_{k}=\alpha _{k+1}$ for some $0\leq k\leq m-1$, so one can always assume
that $\alpha _{0}<\alpha _{1}<\cdots <\alpha _{m}$. \ 

The cases $m=0$ and $m=1$ are straightforward, with canonical completions
given by $\alpha _{0},\alpha _{0},\alpha _{0},\cdots $ and $\alpha
_{0},\alpha _{1},\alpha _{1},\cdots $, respectively. \ The solution of the
case $m=2$ is based on the positivity of the moment matrix $H(1):=$ $\left( 
\begin{array}{cc}
\gamma _{0} & \gamma _{1} \\ 
\gamma _{1} & \gamma _{2}%
\end{array}%
\right) $ and of the localizing matrix $H_{x}(2):=\left( 
\begin{array}{cc}
\gamma _{1} & \gamma _{2} \\ 
\gamma _{2} & \gamma _{3}%
\end{array}%
\right) $; the explicit calculation of the subnormal completion requires 
\textit{recursively generated weighted shifts} \cite[Example 3.12]{RGWSI}. \ 

In the general ($1$-variable) case, the Subnormal Completion Criterion (SCC) %
\cite[Theorem 3.5]{RGWSI} states that a subnormal completion exists if and
only if an $\ell $-hyponormal completion exists, where $\ell :=[\frac{m}{2}%
]+1$.

\begin{theorem}
\label{onevariable}(One-Variable Subnormal Completion Criterion; cf. %
\cite[Theorem 3.5]{RGWSI}) \ Let $\Omega _{m}\equiv \{\alpha
_{k}\}_{k=0}^{m} $ be a finite collection of positive numbers, let $k:=[%
\frac{m+1}{2}]$ and $\ell :=[\frac{m}{2}]+1$, and let $H(k)\equiv H(\Omega
_{m}):=(\gamma _{i+j})_{0\leq i,j\leq k}$, $H_{x}(\ell -1)\equiv
H_{x}(\Omega _{m}):=(\gamma _{i+j+1})_{0\leq i,j\leq \ell -1}$ and $\mathbf{v%
}(i,j):=(\gamma _{i}\;\;\gamma _{i+1}\;\;\cdots \;\;\gamma _{i+j})^{T}$. \
The following statements are equivalent.\newline
(i) $\Omega _{m}$ admits a subnormal completion;\newline
(ii) $\Omega _{m}$ admits an $\ell $-hyponormal completion;\newline
(iii) $H(k)\geq 0$, $H_{x}(\ell -1)\geq 0$, and $\mathbf{v}(k+1,k)\in $ Ran $%
H(k)$ if $m$ is even ($\mathbf{v}(\ell +1,\ell -1)\in $ Ran $H_{x}(\ell -1)$
if $m$ is odd);\newline
(iv) $H(\Omega _{m})$ admits a positive flat (i.e., rank-preserving)
extension $H(\hat{\Omega}_{m+1})$ such that $H_{x}(\hat{\Omega}_{m+1})\geq 0$%
.
\end{theorem}

We now formulate the $2$-variable SCP:

\begin{problem}
\label{SCP}($2$-variable Subnormal Completion Problem) $\ $Given $m\geq 0$
and a finite collection of pairs of positive numbers $\Omega _{m}\equiv
\{(\alpha _{\mathbf{k}},\beta _{\mathbf{k}})\}_{|\mathbf{k}|\leq m}$
satisfying (\ref{commuting}) for all $|\mathbf{k}|\leq m$ (where $|\mathbf{k}%
|:=k_{1}+k_{2}$), find necessary and sufficient conditions to guarantee the
existence of a subnormal $2$-variable weighted shift whose initial weights
are given by $\Omega _{m}$.
\end{problem}

While the research in \cite{RGWSI} provided a complete solution to SCP in
one variable, the $2$-variable version requires new tools and techniques. \
At present, no general solution exists, and the problem appears to be quite
difficult. \ When $m=0$, in one variable the canonical subnormal completion
of $\alpha _{0}$ is the weighted shift $\alpha _{0},\alpha _{0},\alpha
_{0},\ldots $, with \textit{Berger measure} $\mu :=\delta _{\alpha _{0}^{2}}$%
; in two variables, the canonical subnormal completion of $(\alpha
_{00},\beta _{00})$ is the $2$-variable weighted with weight sequences $%
\alpha _{ij}:=\alpha _{00}$ and $\beta _{ij}:=\beta _{00}\;($all $i,j\geq 0)$
and Berger measure $\mu :=\delta _{\alpha _{00}^{2}}\times \delta _{\beta
_{00}^{2}}$. \ 

When $m=1$, the $1$-variable case is still straightforward; i.e., the
canonical subnormal completion is $\alpha _{0},\alpha _{1},\alpha
_{1},\ldots $, with Berger measure $(1-\frac{\alpha _{0}^{2}}{\alpha _{1}^{2}%
})\delta _{0}+\frac{\alpha _{0}^{2}}{\alpha _{1}^{2}}\delta _{\alpha
_{1}^{2}}$. \ In two variables, however, the problem becomes highly
nontrivial. \ For the singular case, and using the results in \cite[Section
6]{tcmp1}, C. Li gave in \cite{Li} a solution, which seems a bit ad hoc and
unmotivated, with extensive calculations using \textit{Mathematica}. \ The
proof in \cite[pages 39 and 40]{tcmp1} establishes the existence of a
representing measure $\mu $ for SCP with quadratic moment data (this is the
case $m=1$ in two variables); however, the ensuing statement that $\operatorname{supp%
}\;\mu \subseteq \mathbb{R}_{+}^{2}$ is made without a proof, and it does
not appear to follow easily from the comments preceding it. \ It is indeed
true, as we show in the present paper using the tools and techniques from %
\cite{tcmp4}.

In Section \ref{Main} below, we shall apply our general strategy to solve
SCP to the case $m=1$ and prove that a representing measure always exist if
the associated moment matrix $M(1)$ is positive semi-definite. \ In Section %
\ref{concrete} we shall calculate the Berger measure using canonical column
relations in the flat extension $M(2)$ of $M(1)$. \ The reader will note how
effective the theory of truncated moment problems can be in detecting the
location of the atoms of the unique representing measure for $M(2)$; this is
in sharp contrast with the ad hoc techniques and extensive symbolic
manipulation present in \cite{Li}.

The case $m=2$, in full generality, remains open; however, in Example \ref%
{exm2} below we solve SCP whenever the associated moment matrix $M(1)$ is
singular. \ For $m\geq 3$, the results in \cite{RGWSI} and \cite{RGWSII}
show that, in the $1$-variable case, it is not always possible to build a
subnormal completion; of course the same is true in two variables: indeed,
if $\alpha _{0},\alpha _{1},\alpha _{2},\alpha _{3}$ is a collection of
weights admitting no subnormal completion, it suffices to consider the $2$%
-variable collection given by $\alpha _{\mathbf{k}}:=\alpha _{k_{1}}$ and $%
\beta _{\mathbf{k}}:=1\;(\left| \mathbf{k}\right| \leq 3)$ in order to
produce such an example. \ 

Problem \ref{SCP} is closely related to truncated moment problems. \ Given
real numbers $\gamma \equiv \gamma ^{(2n)}:=\gamma _{00},\gamma _{01},\gamma
_{10},\gamma _{02},$ $\gamma _{11},\gamma _{20},\cdots ,\gamma _{02n},\cdots
,\gamma _{2n0}$ with $\gamma _{00}>0$, the \textit{truncated real moment
problem} for $\gamma $ entails finding conditions for the existence of a
positive Borel measure $\mu $, supported in $\mathbb{R}^{2}$, such that 
\begin{equation*}
\gamma _{ij}=\int y^{i}x^{j}d\mu ,\quad 0\leq i+j\leq n.
\end{equation*}%
Given $\gamma \equiv \gamma ^{(2n)}$, we can build an associated \textit{%
moment matrix} $M(n)\equiv M(n)(\gamma ):=(M[i,j](\gamma ))_{i,j=0}^{n}$,
where 
\begin{equation*}
M[i,j](\gamma ):=\left( 
\begin{array}{llll}
\gamma _{0,i+j} & \gamma _{1,i+j-1} & \cdots & \gamma _{j,i} \\ 
\gamma _{1,i+j-1} & \gamma _{2,i+j-2} & \cdots & \gamma _{j+1,i-1} \\ 
\text{ \thinspace \thinspace \quad }\vdots & \text{ \thinspace \thinspace
\quad }\vdots & \ddots & \text{ \thinspace \thinspace \quad }\vdots \\ 
\gamma _{i,j} & \gamma _{i+1,j-1} & \cdots & \gamma _{i+j,0}%
\end{array}%
\right) .
\end{equation*}%
We denote the successive rows and columns of $M(n)(\gamma )$ by 
\begin{equation*}
1,X,Y,X^{2},YX,Y^{2},\cdots ,X^{n},\cdots ,Y^{n}.
\end{equation*}%
Observe that each block $M[i,j](\gamma )$ is of \textit{Hankel} form, i.e.,
constant in cross-diagonals. \ (For basic results about truncated moment
problems we refer to \cite{tcmp1} and \cite{tcmp4}.)

We conclude this section by stating a result from \cite{CLY}, which we will
need in Section \ref{Sect3}. \ Recall that a commuting pair $(T_{1},T_{2})$
is $2$\textit{-hyponormal} if the $5$-tuple $%
(T_{1},T_{2},T_{1}^{2},T_{1}T_{2},T_{2}^{2})$ is hyponormal (cf. Section \ref%
{Notation} below). \ For $2$-variable weighted shifts, this is equivalent to
the condition 
\begin{equation*}
M_{\mathbf{u}}(2):=(\gamma _{\mathbf{u}+(m,n)+(p,q)})_{_{0\leq p+q\leq
2}^{0\leq m+n\leq 2}}\geq 0\text{ (all }\mathbf{u}\in \mathbb{Z}_{+}^{2}%
\text{)}\;\text{(cf. \cite[Theorem 2.4]{CLY}),}
\end{equation*}%
that is, 
\begin{equation}
\left( 
\begin{array}{cccccc}
\gamma _{\mathbf{u}} & \gamma _{\mathbf{u}+(0,1)} & \gamma _{\mathbf{u}%
+(1,0)} & \gamma _{\mathbf{u}+(0,2)} & \gamma _{\mathbf{u}+(1,1)} & \gamma _{%
\mathbf{u}+(2,0)} \\ 
\gamma _{\mathbf{u}+(0,1)} & \gamma _{\mathbf{u}+(0,2)} & \gamma _{\mathbf{u}%
+(1,1)} & \gamma _{\mathbf{u}+(0,3)} & \gamma _{\mathbf{u}+(1,2)} & \gamma _{%
\mathbf{u}+(2,1)} \\ 
\gamma _{\mathbf{u}+(1,0)} & \gamma _{\mathbf{u}+(1,1)} & \gamma _{\mathbf{u}%
+(2,0)} & \gamma _{\mathbf{u}+(1,2)} & \gamma _{\mathbf{u}+(2,1)} & \gamma _{%
\mathbf{u}+(3,0)} \\ 
\gamma _{\mathbf{u}+(0,2)} & \gamma _{\mathbf{u}+(0,3)} & \gamma _{\mathbf{u}%
+(1,2)} & \gamma _{\mathbf{u}+(0,4)} & \gamma _{\mathbf{u}+(1,3)} & \gamma _{%
\mathbf{u}+(2,2)} \\ 
\gamma _{\mathbf{u}+(1,1)} & \gamma _{\mathbf{u}+(1,2)} & \gamma _{\mathbf{u}%
+(2,1)} & \gamma _{\mathbf{u}+(1,3)} & \gamma _{\mathbf{u}+(2,2)} & \gamma _{%
\mathbf{u}+(3,1)} \\ 
\gamma _{\mathbf{u}+(2,0)} & \gamma _{\mathbf{u}+(2,1)} & \gamma _{\mathbf{u}%
+(3,0)} & \gamma _{\mathbf{u}+(2,2)} & \gamma _{\mathbf{u}+(3,1)} & \gamma _{%
\mathbf{u}+(4,0)}%
\end{array}%
\right) \geq 0\text{ (all }\mathbf{u}\in \mathbb{Z}_{+}^{2}\text{).}
\label{eq11}
\end{equation}

\medskip An entirely similar formulation exists for $\ell $-hyponormality ($%
\ell \geq 1$), i.e., one requires 
\begin{equation}
M_{\mathbf{u}}(\ell ):=(\gamma _{\mathbf{u}+(m,n)+(p,q)})_{_{0\leq p+q\leq
\ell }^{0\leq m+n\leq \ell }}\geq 0\text{ (all }\mathbf{u}\in \mathbb{Z}%
_{+}^{2}\text{) (cf. \cite[Theorem 2.4]{CLY}).}  \label{mhypo}
\end{equation}

\section{\label{Notation}Notation and Preliminaries}

Let $\mathcal{H}$ be a complex Hilbert space and let $\mathcal{B}(\mathcal{H}%
)$ denote the algebra of bounded linear operators on $\mathcal{H}$. We say
that $T\in \mathcal{B}(\mathcal{H})$ is \textit{normal} if $T^{\ast
}T=TT^{\ast }$, \textit{subnormal} if $T=N|_{\mathcal{H}}$, where $N$ is
normal and $N(\mathcal{H})\mathcal{\subseteq H}$, and \textit{hyponormal} if 
$T^{\ast }T\geq TT^{\ast }$. \ For $S,T\in \mathcal{B}(\mathcal{H})$ let $%
[S,T]:=ST-TS$. \ We say that an $n$-tuple $\mathbf{T}=(T_{1},\cdots ,T_{n})$
of operators on $\mathcal{H}$ is (jointly) \textit{hyponormal} if the
operator matrix 
\begin{equation*}
\lbrack \mathbf{T}^{\ast },\mathbf{T]:=}\left( 
\begin{array}{cccc}
\lbrack T_{1}^{\ast },T_{1}] & [T_{2}^{\ast },T_{1}] & \cdots & [T_{n}^{\ast
},T_{1}] \\ 
\lbrack T_{1}^{\ast },T_{2}] & [T_{2}^{\ast },T_{2}] & \cdots & [T_{n}^{\ast
},T_{2}] \\ 
\vdots & \vdots & \ddots & \vdots \\ 
\lbrack T_{1}^{\ast },T_{n}] & [T_{2}^{\ast },T_{n}] & \cdots & [T_{n}^{\ast
},T_{n}]%
\end{array}%
\right)
\end{equation*}%
is positive semi-definite on the direct sum of $n$ copies of $\mathcal{H}$
(cf. \cite{Ath}, \cite{CMX}). \ The $n$-tuple $\mathbf{T}$ is said to be 
\textit{normal} if $\mathbf{T}$ is commuting and each $T_{i}$ is normal, and 
$\mathbf{T}$ is \textit{subnormal }if $\mathbf{T}$ is the restriction of a
normal $n$-tuple to a common invariant subspace. \ Clearly, normal $%
\Rightarrow $ subnormal $\Rightarrow $ hyponormal.

The Bram-Halmos criterion for subnormality states that an operator $T\in 
\mathcal{B}(\mathcal{H})$ is subnormal if and only if 
\begin{equation*}
\sum_{i,j}(T^{i}x_{j},T^{j}x_{i})\geq 0
\end{equation*}%
for all finite collections $x_{0},x_{1},\cdots ,x_{k}\in \mathcal{H}$ (\cite%
{Bra}, \cite{Con}). \ Using Choleski's algorithm for operator matrices \cite%
{Smu}, it is easy to see that this is equivalent to asserting that the $k$%
-tuple $(T,T^{2},\cdots ,T^{k})$ is hyponormal for all $k\geq 1$.

For $k\geq 1$, we say that a commuting pair $\mathbf{T}\equiv (T_{1},T_{2})$
is $k$\textit{-hyponormal} if $\mathbf{T}%
(k):=(T_{1},T_{2},T_{1}^{2},T_{2}T_{1}$, $T_{2}^{2}$, $\cdots
,T_{1}^{k},T_{2}T_{1}^{k-1},\cdots ,T_{2}^{k})$ is hyponormal (cf. \cite{CLY}%
). \ Clearly, subnormal $\Rightarrow $ $(k+1)$-hyponormal $\Rightarrow $ $k$%
-hyponormal for every $k\geq 1$, and of course $1$-hyponormality agrees with
the usual definition of joint hyponormality. \ The multivariable Bram-Halmos
criterion was obtained in \cite{CLY}, and its formulation is essentially
identical to the $1$-variable one: $\mathbf{T}$ is subnormal if and only if $%
\mathbf{T}(k)$ is hyponormal for all $k\geq 1$.

For $\alpha \equiv \{\alpha _{n}\}_{n=0}^{\infty }$ a bounded sequence of
positive real numbers (called \textit{weights}), let $W_{\alpha }:\ell ^{2}(%
\mathbb{Z}_{+})\rightarrow \ell ^{2}(\mathbb{Z}_{+})$ be the associated 
\textit{unilateral weighted shift}, defined by $W_{\alpha }e_{n}:=\alpha
_{n}e_{n+1}\;($all $n\geq 0)$, where $\{e_{n}\}_{n=0}^{\infty }$ is the
canonical orthonormal basis in $\ell ^{2}(\mathbb{Z}_{+}).$ \ The \textit{%
moments} of $\alpha $ are given as 
\begin{equation*}
\gamma _{k}\equiv \gamma _{k}(\alpha ):=\left\{ 
\begin{array}{cc}
1 & \text{if }k=0 \\ 
\alpha _{0}^{2}\cdots \alpha _{k-1}^{2} & \text{if }k>0%
\end{array}%
\right\} .
\end{equation*}%
It is easy to see that $W_{\alpha }$ is never normal, and that it is
hyponormal if and only if $\alpha _{0}\leq \alpha _{1}\leq \cdots $. \
Similarly, consider double-indexed positive bounded sequences $\alpha \equiv
\{\alpha _{\mathbf{k}}\},\beta \equiv \{\beta _{\mathbf{k}}\}\in \ell
^{\infty }(\mathbb{Z}_{+}^{2})$, $\mathbf{k}\equiv (k_{1},k_{2})\in \mathbb{Z%
}_{+}^{2}:=\mathbb{Z}_{+}\times \mathbb{Z}_{+}$ and let $\ell ^{2}(\mathbb{Z}%
_{+}^{2})$\ be the Hilbert space of square-summable complex sequences
indexed by $\mathbb{Z}_{+}^{2}$. \ (Recall that $\ell ^{2}(\mathbb{Z}%
_{+}^{2})$ is canonically isometrically isomorphic to $\ell ^{2}(\mathbb{Z}%
_{+})\bigotimes \ell ^{2}(\mathbb{Z}_{+})$.) \ We define the $2$-variable
weighted shift $\mathbf{T}\equiv (T_{1},T_{2})$ by 
\begin{equation*}
T_{1}e_{\mathbf{k}}:=\alpha _{\mathbf{k}}e_{\mathbf{k+}\varepsilon _{1}}
\end{equation*}%
\begin{equation*}
T_{2}e_{\mathbf{k}}:=\beta _{\mathbf{k}}e_{\mathbf{k+}\varepsilon _{2}},
\end{equation*}%
where $\mathbf{\varepsilon }_{1}:=(1,0)$ and $\mathbf{\varepsilon }%
_{2}:=(0,1)$. Clearly, 
\begin{equation}
T_{1}T_{2}=T_{2}T_{1}\Longleftrightarrow \beta _{\mathbf{k+}\varepsilon
_{1}}\alpha _{\mathbf{k}}=\alpha _{\mathbf{k+}\varepsilon _{2}}\beta _{%
\mathbf{k}}\;\;(\text{all }\mathbf{k\in }\mathbb{Z}_{+}^{2}).
\label{commuting}
\end{equation}%
In an entirely similar way one can define multivariable weighted shifts.

Given $\mathbf{k}\in \mathbb{Z}_{+}^{2}$, the \textit{moment} of $(\alpha
,\beta )$ of order $\mathbf{k}$ is 
\begin{equation*}
\gamma _{\mathbf{k}}\equiv \gamma _{\mathbf{k}}(\alpha ,\beta ):=\left\{ 
\begin{array}{cc}
1 & \text{if }\mathbf{k}=0 \\ 
\alpha _{(0,0)}^{2}\cdots \alpha _{(k_{1}-1,0)}^{2} & \text{if }k_{1}\geq 1%
\text{ and }k_{2}=0 \\ 
\beta _{(0,0)}^{2}\cdots \beta _{(0,k_{2}-1)}^{2} & \text{if }k_{1}=0\text{
and }k_{2}\geq 1 \\ 
\alpha _{(0,0)}^{2}\cdots \alpha _{(k_{1}-1,0)}^{2}\cdot \beta
_{(k_{1},0)}^{2}\cdots \beta _{(k_{1},k_{2}-1)}^{2} & \text{if }k_{1}\geq 1%
\text{ and }k_{2}\geq 1%
\end{array}%
\right\} .
\end{equation*}%
We remark that, due to the commutativity condition (\ref{commuting}), $%
\gamma _{\mathbf{k}}$ can be computed using any nondecreasing path from $%
(0,0)$ to $(k_{1},k_{2})$.

We also recall a well known characterization of subnormality for
multivariable weighted shifts \cite{JeLu}, due to C. Berger (and
independently to R. Gellar and L.J. Wallen \cite{GeWa}) in the $1$-variable
case: $\ \mathbf{T\equiv (}T_{1},\cdots ,T_{n})$ is subnormal if and only if
there is a probability measure $\mu $ (called the Berger measure of $\mathbf{%
T}$) defined on the $n$-dimensional rectangle $R=[0,a_{1}]\times \cdots
\times \lbrack 0,a_{n}]$ where $a_{i}=\left\| T_{i}\right\| ^{2}$ such that $%
\gamma _{\mathbf{k}}=\int_{R}\mathbf{t}^{\mathbf{k}}d\mu
(t):=\int_{R}t_{1}^{k_{1}}\cdots t_{n}^{k_{n}}d\mu (\mathbf{t})$, for all $%
\mathbf{k\in \mathbb{Z}}_{+}^{n}$. \ 

Consider now a subnormal $1$-variable weighted shift $W_{\alpha }$, with
Berger measure $\xi $, and let $h\geq 1$. \ If we let 
\begin{equation}
\mathcal{M}_{h}:=\bigvee \{e_{n}:n\geq h\}  \label{mh}
\end{equation}
denote the invariant subspace obtained by removing the first $h$ vectors in
the canonical orthonormal basis of $\ell ^{2}(\mathbb{Z}_{+})$, then the
Berger measure of $W_{\alpha }|_{\mathcal{M}_{h}}$ is $\frac{1}{\gamma _{h}}%
t^{h}d\xi (t)$.

An important class of subnormal weighted shifts is obtained by considering
measures $\mu $ with exactly two atoms $t_{0}$ and $t_{1}$. \ These shifts
arise naturally in the Subnormal Completion Problem (\cite{RGWSI}, \cite%
{RGWSII}) and in the theory of truncated moment problems (cf. \cite{Houston}%
, \cite{tcmp1}). \ For $t_{0},t_{1}\in \mathbb{R}_{+}$, $t_{0}<t_{1}$, and $%
\rho _{0},\rho _{1}>0$, the moments of the $2$-atomic measure $\mu :=\rho
_{0}\delta _{t_{0}}+\rho _{1}\delta _{t_{1}}$ (here $\delta _{p}$ denotes
the point-mass probability measure with support the singleton $\{p\}$)
satisfy the $2$-step recursive relation 
\begin{equation}
\gamma _{n+2}=\varphi _{0}\gamma _{n}+\varphi _{1}\gamma _{n+1}\;(n\geq 0);
\label{recrel}
\end{equation}%
at the weight level, this can be written as $\alpha _{n+1}^{2}=\frac{\varphi
_{0}}{\alpha _{n}^{2}}+\varphi _{1}\;(n\geq 0)$. \ The atoms $t_{0}$ and $%
t_{1}$ are the zeros of the generating function 
\begin{equation}
g(t):=t^{2}-\varphi _{1}t-\varphi _{0}.\   \label{generating}
\end{equation}%
More generally, any finitely atomic Berger measure corresponds to a \textit{%
recursively generated} subnormal weighted shift (i.e., one whose moments
satisfy an $r$-step recursive relation); in fact, $r=\operatorname{card}\;\operatorname{supp}
$\ $\mu $. \ In the special case of $r=2$, the theory of recursively
generated weighted shifts makes contact with the work of J. Stampfli in \cite%
{Sta2}, in which he proved that given three positive numbers $\alpha
_{0}<\alpha _{1}<\alpha _{2}$, it is always possible to find a subnormal
weighted shift, denoted $W_{(\alpha _{0},\alpha _{1},\alpha _{2})\symbol{94}%
} $, whose first three weights are $\alpha _{0},\alpha _{1}$ and $\alpha
_{2} $. \ The shift $T\equiv W_{(\alpha _{0},\alpha _{1},\alpha _{2})\symbol{%
94}}$ received special attention in \cite{RGWSII}, and has a $2$-atomic
Berger measure as above; letting $a:=\alpha _{0}^{2}$, $b:=\alpha _{1}^{2}$
and $c:=\alpha _{2}^{2}$, we often refer to this shift as the $abc$ shift. \
We will have occasion to use these shifts in Section \ref{concrete}.

\section{\label{Sect3}Statement of the Subnormal Completion Problem}

\begin{definition}
\label{subnormalcompletion}Given $m\geq 0$ and a finite family of positive
numbers $\Omega _{m}\equiv \{(\alpha _{\mathbf{k}},\beta _{\mathbf{k}%
})\}_{\left| \mathbf{k}\right| \leq m}$, we say that a $2$-variable weighted
shift $\mathbf{T}\equiv (T_{1},T_{2})$ with weight sequences $\alpha _{%
\mathbf{k}}^{\mathbf{T}}$ and $\beta _{\mathbf{k}}^{\mathbf{T}}$ is a
subnormal completion of $\Omega _{m}$ if (i) $\mathbf{T}$ is subnormal, and
(ii) $(\alpha _{\mathbf{k}}^{\mathbf{T}},\beta _{\mathbf{k}}^{\mathbf{T}%
})=(\alpha _{\mathbf{k}},\beta _{\mathbf{k}})$ whenever $\left| \mathbf{k}%
\right| \leq m$.
\end{definition}

\begin{remark}
Note that since a subnormal $2$-variable weighted shift is necessarily
commuting, $\Omega _{m}$ in Definition \ref{subnormalcompletion} satisfies
the commutativity condition in (\ref{commuting}). \ When a family of
positive numbers has this property, we say that it is \textit{commutative}.
\end{remark}

\begin{definition}
Given $m\geq 0$ and a finite family of positive numbers $\Omega _{m}\equiv
\{(\alpha _{\mathbf{k}},\beta _{\mathbf{k}})\}_{\left| \mathbf{k}\right|
\leq m}$, we say that $\hat{\Omega}_{m+1}\equiv \{(\hat{\alpha}_{\mathbf{k}},%
\hat{\beta}_{\mathbf{k}})\}_{\left| \mathbf{k}\right| \leq m+1}$ is an 
\textit{extension} of $\Omega _{m}$ if $(\hat{\alpha}_{\mathbf{k}},\hat{\beta%
}_{\mathbf{k}})=(\alpha _{\mathbf{k}},\beta _{\mathbf{k}})$ whenever $\left| 
\mathbf{k}\right| \leq m$. \ The degree of $\Omega _{m}$, $\deg \;\Omega
_{m} $, is $m+1$. \ When $m=1$, we say that $\Omega _{1}$ is quadratic. \
For $m=2\ell +1$, the moment matrix of $\Omega _{m}$ is 
\begin{equation*}
M(\ell )\equiv M(\Omega _{m})\equiv M_{\mathbf{0}}(\Omega _{m}):=(\gamma
_{(i,j)+(p,q)})_{_{0\leq p+q\leq m}^{0\leq i+j\leq m}}.
\end{equation*}%
Observe that if $\hat{\Omega}_{m+1}$ is commutative, then so is $\Omega _{m}$%
. \ For $m$ odd, $M(\hat{\Omega}_{m+2})$ is an extension of $M(\Omega _{m})$.
\end{definition}

\begin{notation}
When $m=1$, we shall let $a:=\alpha _{00}^{2}$, $b:=\beta _{00}^{2}$, $%
c:=\alpha _{10}^{2}$, $d:=\beta _{01}^{2}$, $e:=\alpha _{01}^{2}$ and $%
f:=\beta _{10}^{2}$. \ To be consistent with the commutativity of a $2$%
-variable weighted shifts whose weight sequences satisfy (\ref{commuting}),
we shall always assume $af=be$. \ The moments of $\Omega _{1}$ are 
\begin{equation*}
\left\{ 
\begin{array}{ccc}
\gamma _{00}:=1 &  &  \\ 
\gamma _{01}:=a & \gamma _{10}:=b &  \\ 
\gamma _{02}:=ac & \gamma _{11}:=be & \gamma _{20}:=bd%
\end{array}%
\right. ,
\end{equation*}%
and the associated moment matrix is 
\begin{equation*}
M(\Omega _{1}):=\left( 
\begin{array}{ccc}
1 & a & b \\ 
a & ac & be \\ 
b & be & bd%
\end{array}%
\right) .
\end{equation*}%
In this case, solving the SCP consists of finding a probability measure $\mu 
$ supported on $\mathbb{R}_{+}^{2}$ such that $\int_{\mathbb{R}%
_{+}^{2}}y^{i}x^{j}\;d\mu (x,y)=\gamma _{ij}\;(i,j\geq 0,\;i+j\leq 2)$. \ 
\end{notation}

Associated with the measure $\mu $ of a subnormal completion is the moment
matrix 
\begin{equation*}
M(2)[\mu ]:=\left( 
\begin{array}{cccccc}
\gamma _{00} & \gamma _{01} & \gamma _{10} & \gamma _{02} & \gamma _{11} & 
\gamma _{20} \\ 
\gamma _{01} & \gamma _{02} & \gamma _{11} & \gamma _{03}[\mu ] & \gamma
_{12}[\mu ] & \gamma _{21}[\mu ] \\ 
\gamma _{10} & \gamma _{11} & \gamma _{20} & \gamma _{12}[\mu ] & \gamma
_{21}[\mu ] & \gamma _{30}[\mu ] \\ 
\gamma _{02} & \gamma _{03}[\mu ] & \gamma _{12}[\mu ] & \gamma _{04}[\mu ]
& \gamma _{13}[\mu ] & \gamma _{22}[\mu ] \\ 
\gamma _{11} & \gamma _{12}[\mu ] & \gamma _{21}[\mu ] & \gamma _{13}[\mu ]
& \gamma _{22}[\mu ] & \gamma _{31}[\mu ] \\ 
\gamma _{20} & \gamma _{21}[\mu ] & \gamma _{30}[\mu ] & \gamma _{22}[\mu ]
& \gamma _{31}[\mu ] & \gamma _{40}[\mu ]%
\end{array}%
\right) \text{ \ (cf. (\ref{eq11})).}
\end{equation*}%
The (quartic) moments of $\mu $ give rise to an extension $\hat{\Omega}_{3}$
of $\Omega _{1}$, so that $M(2)[\mu ]=M(\hat{\Omega}_{3})$. \ It is thus
clear that a necessary condition for the existence of a measure $\mu $ is
the positivity of $M(\hat{\Omega}_{3})$, which in turn implies the
positivity of $M(\Omega _{1})$. $\ $If we now let $p:=\hat{\alpha}_{20}^{2}$%
, $q:=\hat{\alpha}_{11}^{2}$, $r:=\hat{\alpha}_{02}^{2}$ and $s:=\hat{\beta}%
_{02}^{2}$, we see that 
\begin{equation*}
M(\hat{\Omega}_{3}):=\left( 
\begin{array}{cccccc}
1 & a & b & ac & be & bd \\ 
a & ac & be & acp & beq & bdr \\ 
b & be & bd & beq & bdr & bds \\ 
ac & acp & beq & \gamma _{04}[\mu ] & \gamma _{13}[\mu ] & \gamma _{22}[\mu ]
\\ 
be & beq & bdr & \gamma _{13}[\mu ] & \gamma _{22}[\mu ] & \gamma _{31}[\mu ]
\\ 
bd & bdr & bds & \gamma _{22}[\mu ] & \gamma _{31}[\mu ] & \gamma _{40}[\mu ]%
\end{array}%
\right) .
\end{equation*}%
The localizing matrices $M_{x}(\hat{\Omega}_{3})$ and $M_{y}(\hat{\Omega}%
_{3})$ (cf. \cite[Introduction]{tcmp4}) are 
\begin{equation*}
M_{x}(\hat{\Omega}_{3})=\left( 
\begin{array}{ccc}
a & ac & be \\ 
ac & acp & beq \\ 
be & beq & bdr%
\end{array}%
\right) \text{ and }M_{y}(\hat{\Omega}_{3})=\left( 
\begin{array}{ccc}
b & be & bd \\ 
be & beq & bdr \\ 
bd & bdr & bds%
\end{array}%
\right) .
\end{equation*}%
(The matrix $M_{x}(\hat{\Omega}_{3})$ is the compression of $M(\hat{\Omega}%
_{3})$ to the first three rows and to the columns indexed by monomials
containing $X$, that is, $X$, $X^{2}$ and $YX$; the matrix $M_{y}(\hat{\Omega%
}_{3})$ is defined similarly.) \ Observe that $M_{x}(\hat{\Omega}%
_{3})=M_{(0,1)}(1)$ and $M_{y}(\hat{\Omega}_{3})=M_{(1,0)}(1)$ (cf. (\ref%
{eq11})). \ For the existence of a measure $\mu $ supported in $\mathbb{R}%
_{+}^{2}$, it is necessary to have $M_{x}(\hat{\Omega}_{3})\geq 0$ and $%
M_{y}(\hat{\Omega}_{3})\geq 0$. \ 

In this paper we prove that starting with the positivity of $M(\Omega _{1})$
alone, it is possible to choose new weights $p$, $q$, $r$ and $s$ to ensure
the positivity of $M_{x}(\hat{\Omega}_{3})$ and $M_{y}(\hat{\Omega}_{3})$. \
We can do this while simultaneously building a positive flat moment matrix
extension $M(\hat{\Omega}_{3})$ of $M(\Omega _{1})$. \ Once we establish the
simultaneous positivity of $M(\hat{\Omega}_{3})$, $M_{x}(\hat{\Omega}_{3})$
and $M_{y}(\hat{\Omega}_{3})$, the existence of a representing measure $\mu $
follows from the main result in \cite{tcmp4}. \ We prove this in Section \ref%
{Main}. \ In Section \ref{concrete} we give a concrete description of $\mu $
in terms of the initial data $a$, $b$, $c$, $d$ and $e$ and the new weights $%
p$, $q$, $r$ and $s$. \ First, we present in Section \ref{abstract} an
abstract solution to SCP, which uses our new approach, involving localizing
matrices and the results in \cite{tcmp4}.

While the flat extension approach is successful in the case $m=1$, it will
not lead to a solution of SCP in all cases. \ Indeed, it is possible to
build a moment matrix $M(2)\equiv M(\Omega _{3})$ admitting a representing
measure, but with no flat extension $M(3)$ (cf. Section \ref{flat} below). \
This shows that our approach, while very general, will not yield subnormal
completions merely by one-step flat extension techniques. \ In many
instances, solving SCP will require a finite sequence of rank-increasing
extensions followed by a flat extension; this is despite the fact that for
SCP one looks for a measure with support in the nonnegative quarter-plane. \
As a matter of fact, the ``translation of support'' technique we use in
Section \ref{flat} shows that solving SCP is equivalent to solving $K$-TMP,
where $K$ is a compact set satisfying $K\subseteq \mathbb{R}_{+}^{2}$ and $%
K\cap (0,+\infty )^{2}\neq \varnothing $. \ Thus, SCP is a special case of $%
K $-TMP, and it is natural to expect that qualitative aspects of TMP theory
will be appropriately reflected in SCP.

\section{\label{abstract}Abstract Solution of SCP}

In this section we will give an abstract solution of Problem \ref{SCP}. \ We
first consider the main theorem in \cite{tcmp4}. \ Although \cite[Theorem
1.6]{tcmp4} deals with truncated complex moment problems, there is an
entirely equivalent version for the case of two real variables, which we now
state.

\begin{theorem}
\label{tcmp4}Let $\mathcal{P}\equiv \left\{ p_{1},\dots ,p_{N}\right\}
\subseteq \mathbb{C}\left[ x,y\right] $ and define $k_{i}$ by $\deg
\;p_{i}=2k_{i}$ or $\deg \;p_{i}=2k_{i}-1$ ($1\leq i\leq N$). \ There exists
a $\operatorname{rank}\;M\left( n\right) $-atomic representing measure for $\gamma
^{\left( 2n\right) }$ supported in $K_{\mathcal{P}}:=\left\{ (x,y)\in 
\mathbb{R}^{2}:p_{i}\left( x,y\right) \geq 0,\;1\leq i\leq N\right\} $ if
and only if $M\left( n\right) \geq 0$ and there is some flat extension $%
M\left( n+1\right) $ for which $M_{p_{i}}\left( n+k_{i}\right) \geq 0$ 
\textup{(}$1\leq i\leq N$\textup{)}. \ In this case, the representing
measure for $M\left( n+1\right) $ is $\operatorname{rank}\;M\left( n\right) $%
-atomic, supported in $K_{\mathcal{P}}$, and with precisely $\operatorname{rank}%
\;M(n)-\operatorname{rank}\;M_{p_{i}}(n+k_{i})$ atoms in $\mathcal{Z}\left(
p_{i}\right) :=\{(x,y)\in \mathbb{R}^{2}:p_{i}(x,y)=0\}$ \textup{(}$1\leq
i\leq N$\textup{)}.
\end{theorem}

With the aid of Theorem \ref{tcmp4}, we can now state and prove a result
which gives a sufficient condition for the solubility of SCP in two
variables. \ Our version does not completely match the conditions listed on
Theorem \ref{onevariable}, and we now explain why. \ In one variable,
building a flat moment matrix extension of a Hankel matrix entails adding an
extra row and an extra column, and checking that the rank is preserved. \
This entails checking the range condition in Theorem \ref{onevariable}(iii)
and ensuring that the new lower right-hand corner entry satisfies the
requirement in Smul'jan's Lemma \cite{Smu}:

\begin{lemma}
\label{smu}(cf. \cite[Proposition 2.2]{tcmp1}) \ Consider the $2\times 2$
block matrix $D:=\left( 
\begin{array}{cc}
A & B \\ 
B^{\ast } & C%
\end{array}%
\right) $. \ Then 
\begin{equation*}
D\geq 0\Longleftrightarrow A\geq 0\text{, }B=AW\text{ for some }W\text{, and 
}C\geq W^{\ast }AW\text{.}
\end{equation*}
\end{lemma}

In two variables, what one adds is not a row and a column but instead a 
\textit{block} of rows and a \textit{block} of columns; while it is still
possible to preserve the range condition, the new lower right-hand corner is
not a number but a square matrix, which must necessarily be Hankel for the
extension to be a moment matrix. \ One easily finds out that $\ell $%
-hyponormality (cf. Theorem \ref{onevariable}(ii)), while necessary, is no
longer sufficient to prove the hankelicity of the new lower right-hand
block. \ Thus, our result avoids mention of $\ell $-hyponormality. \
Moreover, solving the SCP admits two structurally different cases: $m$ odd
and $m$ even. \ In the former case, $\deg \;\Omega _{m}\;(=m+1)$ is even, so
we have enough moments to build the moment matrix $M(\Omega _{m})$. \ 

The same is not true, however, when $m=2k$, since we have moments up to
degree $2k+1$, and this does not allow us to build a complete moment matrix.
\ In the terminology of Lemma \ref{smu}, we have $A:=M(\Omega _{m-1})$, and
also the $B$ block (consisting of moments up to degree $m+1$), but no $C$
block. \ Since we are seeking a moment matrix $M(\Omega _{m+1})$, with
moments up to degree $2m+2$, we can certainly require that Ran $B\subseteq $
Ran $A\equiv $ Ran $M(\Omega _{m-1})$, but that in itself does not generate
the additional moments. \ One could attempt to define the $C$ block as $%
W^{\ast }AW$ (where $W$ solves the equation $AW=B$), but this in general
does not produce a Hankel block $C$, as has been observed in \cite{tcmp6}. \
Therefore, it becomes necessary to postulate the existence of moments of
degree $m+1$ that, together with the initial data $\Omega _{m}$, allows us
to build a moment matrix, which we will call $M(\Omega _{m+1})$. \ 

\begin{theorem}
\label{main}Let $\Omega _{m}:=\{(\alpha _{\mathbf{k}},\beta _{\mathbf{k}}):|%
\mathbf{k}|\leq m\}$ be an initial set of positive weights satisfying the
commutativity condition $\beta _{\mathbf{k+}\varepsilon _{1}}\alpha _{%
\mathbf{k}}=\alpha _{\mathbf{k+}\varepsilon _{2}}\beta _{\mathbf{k}}\;\;($%
all $\mathbf{k}\in \mathbb{Z}_{+}^{2}$ with $\left| \mathbf{k+}\varepsilon
_{i}\right| \leq m\;(i=1,2))$, and let $\tilde{m}:=2\left[ \frac{m}{2}\right]
+1$; thus $\tilde{m}=m$ if $m$ is odd and $\tilde{m}=m+1$ if $m$ is even. \
Assume that $M(\Omega _{\tilde{m}})\geq 0$, and that $\Omega _{\tilde{m}}$
admits a commutative extension $\hat{\Omega}_{\tilde{m}+2}$ such that the
moment matrix $M(\hat{\Omega}_{\tilde{m}+2})$ is a flat (i.e.,
rank-preserving) extension of $M(\Omega _{\tilde{m}})$, with $M_{x}(\hat{%
\Omega}_{\tilde{m}+2})\geq 0$ and $M_{y}(\hat{\Omega}_{\tilde{m}+2})\geq 0$.
\ Then there exists a $\operatorname{rank}\;M\left( \Omega _{\tilde{m}}\right) $%
-atomic representing measure $\mu $ supported in $\mathbb{R}_{+}^{2}$, with
precisely $\operatorname{rank}\;M(\Omega _{\tilde{m}})-\operatorname{rank}\;M_{x}(\hat{\Omega%
}_{\tilde{m}+2})$ atoms in $\{0\}\times \mathbb{R}_{+}\,$(resp. $\operatorname{rank}%
\;M(\Omega _{\tilde{m}})-\operatorname{rank}\;M_{y}(\hat{\Omega}_{\tilde{m}+2})$
atoms in $\mathbb{R}_{+}\times \{0\}$). $\ $The measure $\mu $ is the Berger
measure of a subnormal completion $\hat{\Omega}_{\infty }$ of $\Omega _{m}$,
provided at least one atom of $\mu $ lies inside the positive quadrant in $%
\mathbb{R}^{2}$.
\end{theorem}

\begin{proof}
In the case at hand, the polynomials $p_{i}$ are $p_{1}(x,y):=x$ and $%
p_{2}(x,y):=y$; thus, $k_{1}=k_{2}=1$. \ It follows that $K_{\mathcal{P}}=%
\mathbb{R}_{+}^{2}$ and that $M_{p_{1}}\left( n+k_{1}\right) =M_{x}(n+1)$
and $M_{p_{2}}\left( n+k_{2}\right) =M_{y}(n+1)$. \ Our result now follows
from a straightforward application of Theorem \ref{tcmp4}.
\end{proof}

Despite its simplicity, Theorem \ref{main} is quite useful, as we will see
in the next section. \ We conclude this section by showing how the
additional moments required in case $m$ is even are sometimes determined by $%
M(\Omega _{m-1})$.

\begin{example}
\label{exm2}Let $m=2$ and assume that $A:=M(\Omega _{1})\geq 0$ and $\det
A=0 $. \ Then there exist moments $\gamma _{i,j}\;(i+j=4)$ such that $%
M(\Omega _{3})\geq 0$ is a flat extension of $A$. $\ $The case when $\operatorname{%
rank}\;A=1$ is easily disposed of, so without loss of generality we focus on
the case $Y=a1+bX$ in the column space of $A$. \ We are assuming that $A\geq
0$, $M_{x}(\Omega _{3})\geq 0$, $M_{y}(\Omega _{3})\geq 0$ and Ran $%
B\subseteq $ Ran $A$. \ (Observe that $M_{x}(\Omega _{3})$ and $M_{y}(\Omega
_{3})$ include moments up to degree $3$, so building them requires no new
moments.) \ The equation $\det \;A=0$ uniquely determines $\gamma _{02}$,
from which we obtain at once the weight 
\begin{equation*}
\beta _{01}=\frac{\alpha _{00}^{2}\beta _{00}^{2}\alpha _{10}^{2}-2\alpha
_{00}^{2}\beta _{00}^{2}\alpha _{01}^{2}+\beta _{00}^{2}\alpha _{01}^{4}}{%
\alpha _{00}^{2}(\alpha _{10}^{2}-\alpha _{00}^{2})}.\ 
\end{equation*}%
Since Ran $B\subseteq $ Ran $A$, each column in $B$ must be a linear
combination of the columns $1$ and $X$, and straightforward calculations
using \textit{Mathematica} yield unique values for $\alpha _{20}$, $\alpha
_{11}$ and $\alpha _{02}$. \ Concretely, 
\begin{eqnarray*}
\alpha _{20}^{2} &=&\frac{\alpha _{00}^{2}\alpha _{10}^{4}-\alpha
_{00}^{2}\alpha _{10}^{2}\alpha _{01}^{2}+\alpha _{00}^{2}\alpha
_{01}^{2}\alpha _{11}^{2}-\alpha _{10}^{2}\alpha _{01}^{2}\alpha _{11}^{2}}{%
\alpha _{10}^{2}(\alpha _{00}^{2}-\alpha _{01}^{2})} \\
\alpha _{11}^{2} &=&\frac{\alpha _{00}^{2}\alpha _{10}^{2}\alpha
_{01}^{2}-\alpha _{00}^{2}\alpha _{01}^{4}-\alpha _{00}^{2}\alpha
_{10}^{2}\alpha _{02}^{2}+2\alpha _{00}^{2}\alpha _{01}^{2}\alpha
_{02}^{2}-\alpha _{01}^{4}\alpha _{02}^{2}}{\alpha _{01}^{2}(\alpha
_{00}^{2}-\alpha _{01}^{2})} \\
\alpha _{02}^{2} &=&\frac{\alpha _{00}^{2}(\beta _{00}^{2}\alpha
_{10}^{2}-\beta _{00}^{2}\alpha _{01}^{2}+\alpha _{00}^{2}\beta
_{02}^{2}-\alpha _{10}^{2}\beta _{02}^{2})}{\beta _{00}^{2}(\alpha
_{00}^{2}-\alpha _{01}^{2})}.
\end{eqnarray*}%
With this information at our disposal, it is now straightforward to check
that the $C$ block, defined as $C:=mW^{\ast }AW$ (where $AW=B)$ is Hankel. \
Thus, $M(\Omega _{3}):=\left( 
\begin{array}{cc}
A & B \\ 
B^{\ast } & C%
\end{array}%
\right) $ is a moment matrix extension of $A$, and moreover $\operatorname{rank}%
\;M(\Omega _{3})=\operatorname{rank}\;A=2$. \ It is now clear that SCP admits a
solution in this particular case. \ \newline
One might wish to extend the above reasoning to the case $\operatorname{rank}\;A=3$,
as follows. \ Let $W:=A^{-1}B$ and let $C:=W^{\ast }AW$. \ It is well known
that $C$ is in general not Hankel, and that one can make it Hankel by adding
a rank-one positive matrix $P$. \ Thus, $M(\Omega _{3}):=\left( 
\begin{array}{cc}
A & B \\ 
B^{\ast } & C+P%
\end{array}%
\right) $ is a positive moment matrix, and $\operatorname{rank}\;M(\Omega _{3})=4$.
\ The solution of the Quartic Moment Problem \cite{tcmp6} now says that
there exists a flat extension $M(\Omega _{5})$ of $M(\Omega _{3})$. \
Unfortunately, we can't tell whether the support of the representing measure
for $M(\Omega _{5})$ is contained in the first quadrant in $\mathbb{R}^{2}$.
\ This would require verifying that the localizing matrices $M_{x}(\Omega
_{5})$ and $M_{y}(\Omega _{5})$ are positive. \ If we knew that they are
flat extensions of $M_{x}(\Omega _{3})$ and $M_{y}(\Omega _{3})$, resp.,
then of course we would be done. \ This fact is false in general, but it
might be true in the context of SCP; however, we have not been able to prove
it for SCP.
\end{example}

\section{\label{Main}Localizing Matrices as Flat Extension Builders}

We now specialize to the case $m=1$ in two variables, and show that the
condition $M(\Omega _{1})\geq 0$ is sufficient for the existence of a
subnormal completion.

\begin{theorem}
\label{quartic}Let $\Omega _{1}$ be a quadratic, commutative, initial set of
positive weights, and assume $M(\Omega _{1})\geq 0$. \ Then there always
exists a quartic commutative extension $\hat{\Omega}_{3}$ of $\Omega _{1}$
such that $M(\hat{\Omega}_{3})$ is a flat extension of $M(\Omega _{1})$, and 
$M_{x}(\hat{\Omega}_{3})\geq 0$ and $M_{y}(\hat{\Omega}_{3})\geq 0$. \ As a
consequence, $\Omega _{1}$ admits a subnormal completion $\mathbf{T}_{\hat{%
\Omega}_{\infty }}$.
\end{theorem}

\begin{proof}
Since $m=1$, we have $\ell =1$. \ By Theorem \ref{main}, we first need to
show that six new weights, $\hat{\alpha}_{20},\hat{\beta}_{20},\hat{\alpha}%
_{11},\hat{\beta}_{11},\hat{\alpha}_{02}$ and $\hat{\beta}_{02}$ can be
chosen in such a way that $M_{x}(\hat{\Omega}_{3})\geq 0$ and $M_{y}(\hat{%
\Omega}_{3})\geq 0$. \ Once we prove this, we shall employ techniques from
truncated moment problems to establish the existence of a flat extension $M(%
\hat{\Omega}_{3})$ of $M(\Omega _{1})$. \ We will then appeal to the main
result in \cite{tcmp4}; the existence of a flat extension will readily imply
the existence of a representing measure $\mu $ for $M(1)$, and the
positivity of the localizing matrices $M_{x}(2)$ and $M_{y}(2)$ means that $%
\operatorname{supp}\;\mu \subseteq \mathbb{R}_{+}^{2}$. \ Thus, $\mu $ will be the
Berger measure of a subnormal $2$-variable weighted shift $\mathbf{T}%
_{\Omega _{\infty }}$, which will be the desired subnormal completion of $%
\Omega _{1}$. \ 

We now build $M(2)$. \ To simplify the calculations, we let 
\begin{equation*}
\left\{ 
\begin{array}{cc}
a:=\alpha _{00}^{2} & b:=\beta _{00}^{2} \\ 
c:=\alpha _{10}^{2} & d:=\beta _{01}^{2} \\ 
e:=\alpha _{01}^{2} & f:=\beta _{10}^{2}%
\end{array}%
\right. .
\end{equation*}%
(The family $\Omega _{1}$ is shown in Figure \ref{initial}.)

\setlength{\unitlength}{1mm} \psset{unit=1mm} 
\begin{figure}[th]
\begin{center}
\begin{picture}(50,45)

\psline(0,0)(40,0)
\psline(0,20)(20,20)
\psline(0,0)(0,40)
\psline(20,0)(20,20)

\put(8,2){\footnotesize{$\sqrt{a}$}}
\put(28,2){\footnotesize{$\sqrt{c}$}}

\put(8,22){\footnotesize{$\sqrt{e}$}}

\put(1,9){\footnotesize{$\sqrt{b}$}}
\put(1,29){\footnotesize{$\sqrt{d}$}}
\put(21,9){\footnotesize{$\sqrt{f}$}}

\end{picture}
\end{center}
\caption{The initial family of weights $\Omega _{1}$}
\label{initial}
\end{figure}
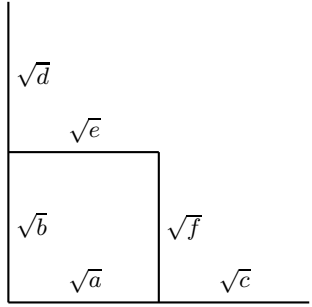

Thus, 
\begin{equation}
M(1)=\left( 
\begin{array}{ccc}
1 & a & b \\ 
a & ac & be \\ 
b & be & bd%
\end{array}%
\right) .  \label{m1}
\end{equation}%
Since $M(1)\geq 0$, it follows that $\det \;\left( 
\begin{array}{cc}
ac & be \\ 
be & bd%
\end{array}%
\right) \geq 0$, i.e., 
\begin{equation}
acd\geq be^{2}.  \label{weight1}
\end{equation}%
\ By the commutativity of $\Omega _{1}$, we have 
\begin{equation}
af=be,  \label{af}
\end{equation}%
and therefore 
\begin{equation}
cd\geq ef.  \label{ineq1}
\end{equation}%
\ A straightforward calculation shows that 
\begin{equation*}
\det \;M(1)=acbd-b^{2}e^{2}-a^{2}bd+2ab^{2}e-b^{2}ac
\end{equation*}%
and that 
\begin{equation}
\det \;M(1)>0\Longrightarrow cd-ef>0;  \label{cdef}
\end{equation}%
for, if $cd-ef=0$ then the rank of the $2\times 2$ lower right-hand corner
of $M(1)$ is $1$, and then $M(1)$ cannot be invertible. \ Inspection of (\ref%
{ineq1}) reveals that we must have $c\geq e$ or $d\geq f$. \ Without loss of
generality, we shall assume that $c\geq e$. \ We also assume that $a<c$,
since otherwise a trivial solution exists. \ (In fact, if $a=c$ in (\ref{m1}%
), the positivity of $M(1)$ implies that $a=e$ and $b=f\leq d$; when $b=d$
(resp. $b<d$), the point mass $\delta _{(a,b)}$ is the Berger measure of the
subnormal completion (resp. $(1-\frac{b}{d})\delta _{(a,0)}+\frac{b}{d}%
\delta _{(a,d)}$). \ Thus, in what follows we shall always assume $c\geq e$
and $a<c$.

To build $M(2)\equiv M(\hat{\Omega}_{3})$, we first need six new weights
(the quadratic weights), namely $\hat{\alpha}_{20}$, $\hat{\beta}_{20}$, $%
\hat{\alpha}_{11}$, $\hat{\beta}_{11}$, $\hat{\alpha}_{02}$ and $\hat{\beta}%
_{02}$. \ Since the extension $\hat{\Omega}_{3}$ will also be commutative,
two of these weights will be expressible in terms of other weights. \ We
thus denote $\hat{\alpha}_{20}$ by $\sqrt{p}$, $\hat{\alpha}_{11}$ by $\sqrt{%
q}$, $\hat{\alpha}_{02}$ by $\sqrt{r}$, and $\hat{\beta}_{02}$ by $\sqrt{s}$
($\hat{\beta}_{20}$ and $\hat{\beta}_{11}$ can be written in terms of the
other four new weights). \ It follows that 
\begin{equation}
M(2)=\left( 
\begin{array}{cccccc}
1 & a & b & ac & be & bd \\ 
a & ac & be & acp & beq & bdr \\ 
b & be & bd & beq & bdr & bds \\ 
ac & acp & beq &  &  &  \\ 
be & beq & bdr &  &  &  \\ 
bd & bdr & bds &  &  & 
\end{array}%
\right)  \label{incomplete}
\end{equation}%
(with the lower right-hand $3\times 3$ corner yet undetermined) and 
\begin{equation*}
M_{x}(2)=\left( 
\begin{array}{ccc}
a & ac & be \\ 
ac & acp & beq \\ 
be & beq & bdr%
\end{array}%
\right) \text{ and }M_{y}(2)=\left( 
\begin{array}{ccc}
b & be & bd \\ 
be & beq & bdr \\ 
bd & bdr & bds%
\end{array}%
\right) .
\end{equation*}%
Now, since the zero-th row of a subnormal completion of $\Omega _{1}$ will
be a subnormal completion of the zero-th row of $\Omega _{1}$, which is
given by the weights $a\leq c$, we let $p:=c$. \ By one of the main results
in \cite{propagation}, having $\alpha _{10}=\hat{\alpha}_{20}$ immediately
implies that $\hat{\alpha}_{11}=\sqrt{c}$, that is, $q:=c$. \ Thus, 
\begin{equation*}
M_{x}(2)=\left( 
\begin{array}{ccc}
a & ac & be \\ 
ac & ac^{2} & bce \\ 
be & bce & bdr%
\end{array}%
\right) .
\end{equation*}%
By Choleski's Algorithm \cite{Atk}, $M_{x}(2)\geq 0$ if and only if $bdr\geq 
\frac{(be)^{2}}{a}$, so that we need $r\geq \frac{ef}{d}$. \ Thus, provided
we take $r\geq \frac{ef}{d}$, the positivity of $M_{x}(2)$ is guaranteed. \
It remains to show that we can choose $s$ in such a way that $s\geq d$ and $%
M_{y}(2)\equiv M_{y}(2)(s)\geq 0$. We consider two cases.

\textbf{Case 1}: $e=c$. \ By (\ref{ineq1}) we have $d\geq f$, so we can take 
$r:=c$ and guarantee that $M_{x}(2)\geq 0$. \ We also let $s:=d$. \ We then
have 
\begin{equation*}
M_{y}(2)=\left( 
\begin{array}{ccc}
b & bc & bd \\ 
bc & bc^{2} & bcd \\ 
bd & bcd & bd^{2}%
\end{array}%
\right) .
\end{equation*}%
It follows at once that $\operatorname{rank}\;M_{y}(2)=1$, and therefore $%
M_{y}(2)\geq 0$ (and of course $s\geq d$).

\textbf{Case 2}: $e<c$. \ We define $r$ by this extremal value, i.e., $r:=%
\frac{ef}{d}$. \ This immediately implies that $\hat{\beta}_{11}:=\sqrt{f}$,
and by propagation, $\hat{\beta}_{1j}:=\sqrt{f}$ (all $j\geq 2$) in any
subnormal completion. \ The resulting weight diagram is shown in Figure \ref%
{weight}.

\setlength{\unitlength}{1mm} \psset{unit=1mm} 
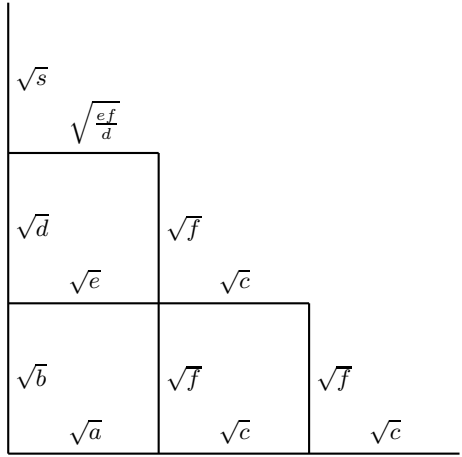
\begin{figure}[h]
\begin{center}
\begin{picture}(70,65)

\psline(0,0)(60,0)
\psline(0,20)(40,20)
\psline(0,40)(20,40)
\psline(0,0)(0,60)
\psline(20,0)(20,40)
\psline(40,0)(40,20)

\put(8,2){\footnotesize{$\sqrt{a}$}}
\put(28,2){\footnotesize{$\sqrt{c}$}}
\put(48,2){\footnotesize{$\sqrt{c}$}}

\put(8,22){\footnotesize{$\sqrt{e}$}}
\put(28,22){\footnotesize{$\sqrt{c}$}}

\put(8,43){\footnotesize{$\sqrt{\frac{ef}{d}}$}}

\put(1,9){\footnotesize{$\sqrt{b}$}}
\put(1,29){\footnotesize{$\sqrt{d}$}}
\put(1,49){\footnotesize{$\sqrt{s}$}}

\put(21,9){\footnotesize{$\sqrt{f}$}}
\put(21,29){\footnotesize{$\sqrt{f}$}}

\put(41,9){\footnotesize{$\sqrt{f}$}}

\end{picture}
\end{center}
\caption{The family $\Omega _{1}$ augmented with the inclusion of the
quadratic weights}
\label{weight}
\end{figure}

It remains to define $s$, in such a way that $s\geq d$ and $M_{y}(2)\geq 0$.
\ Since 
\begin{equation*}
M_{y}(2)\equiv M_{y}(2)(s)=\left( 
\begin{array}{ccc}
b & be & bd \\ 
be & bce & bef \\ 
bd & bef & bds%
\end{array}%
\right)
\end{equation*}%
and the $2\times 2$ upper left-hand corner of $M_{y}(2)$ is invertible, we
see that $M_{y}(2)\geq 0$ if and only if $\det \;M_{y}(2)(s)\geq 0$. \ Since 
$\det \;M_{y}(2)(s)$ is linear in $s$, we pick for $s$ the unique value that
makes $\det \;M_{y}(2)(s)=0$. \ A straightforward calculation shows that 
\begin{equation*}
s=\frac{a^{2}cd^{2}-2abde^{2}+b^{2}e^{3}}{a^{2}d(c-e)}.\ 
\end{equation*}%
We then have 
\begin{equation*}
s-d=\frac{e(ad-be)^{2}}{a^{2}d(c-e)}\geq 0.
\end{equation*}%
Thus, this particular choice of $s$ guarantees both $s\geq d$ and $%
M_{y}(2)\geq 0$.

To complete the proof, we need to define the $3\times 3$ lower right-hand
corner of $M(2)$, and then show that $M(2)$ is a flat extension of $M(1)$,
and therefore $M(2)\geq 0$. \ We consider the following two cases.

(i) $\operatorname{rank}\;M(1)=2$: Without loss of generality, we may assume that $%
a<c$, so that the columns $1$ and $X$ of $M(1)$ are linearly independent. \
The column $Y$ must then be a linear combination of $1$ and $X$, and that
allows us to define $YX$ and $Y^{2}$ in $M(2)$. \ Moreover, since the
zero-th row of $\mathbf{T}_{\hat{\Omega}_{\infty }}$ is given by the weights 
$\sqrt{a},\sqrt{c},\sqrt{c},\cdots $, whose Berger measure is $\xi _{x}=(1-%
\frac{a}{c})\delta _{0}+\frac{a}{c}\delta _{c}$ (and thus supported in the
two-point set $\{0,c\}$), it is natural to let $X^{2}:=cX$ in the column
space of $M(2)$. \ With these definitions, one easily verifies that the
truncations to the first three rows of $X^{2}$, $YX$ and $Y^{2}$ agree with
the $3\times 3$ upper right-hand corner of the matrix $M(2)$ in (\ref%
{incomplete}). \ It is clear that the matrix $M(2)$ thus defined is positive
semi-definite, but one needs to verify that $M(2)$ is a moment matrix. \
This amounts to checking that the $(4,6)$ and $(5,5)$ entries are equal. \
Now, a straightforward calculation shows that in the column space of $M(1)$
we have 
\begin{equation}
Y=\frac{b(c-e)}{c-a}\cdot 1+\frac{f-b}{c-a}X,  \label{y}
\end{equation}%
so that 
\begin{eqnarray*}
M(2)_{46} &=&\left\langle Y^{2},X^{2}\right\rangle =\left\langle
Y^{2},cX\right\rangle \\
&=&c\left\langle Y,YX\right\rangle =c\left\langle \frac{b(c-e)}{c-a}\cdot 1+%
\frac{f-b}{c-a}X,YX\right\rangle \\
&=&c\frac{b(c-e)}{c-a}be+c\frac{f-b}{c-a}bce \\
&=&bce\frac{cf-be}{c-a}=bcef.
\end{eqnarray*}%
On the other hand, using (\ref{y}) we define $YX:=\frac{b(c-e)}{c-a}X+\frac{%
f-b}{c-a}X^{2}$, so that 
\begin{eqnarray*}
M(2)_{55} &=&\left\langle YX,YX\right\rangle \\
&=&\left\langle \frac{b(c-e)}{c-a}X+\frac{f-b}{c-a}X^{2},YX\right\rangle \\
&=&\frac{b(c-e)}{c-a}bce+\frac{f-b}{c-a}\left\langle cX,YX\right\rangle \\
&=&\frac{b(c-e)}{c-a}bce+c\frac{f-b}{c-a}bce \\
&=&bce\frac{cf-be}{c-a}=bcef.
\end{eqnarray*}%
It follows that $M(2)_{46}=M(2)_{55}$, as desired. \ In this case, the
representing measure is supported in the two-point set $%
\{(0,y_{0}),(c,y_{c})\}$, where 
\begin{equation}
y_{0}:=\frac{b(c-e)}{c-a}  \label{y0}
\end{equation}%
and 
\begin{equation}
y_{c}:=\frac{b(c-e)}{c-a}+\frac{f-b}{c-a}c=\frac{cf-be}{c-a}=f.  \label{yc}
\end{equation}

(iii) $\operatorname{rank}\;M(1)=3$: We let $B$ denote the upper right-hand corner
of $M(2)$, that is, 
\begin{equation*}
B:=\left( 
\begin{array}{ccc}
ac & be & bd \\ 
acp & beq & bdr \\ 
beq & bdr & bds%
\end{array}%
\right) =\left( 
\begin{array}{ccc}
ac & be & bd \\ 
ac^{2} & bce & bdr \\ 
bce & bdr & bds%
\end{array}%
\right) .
\end{equation*}%
We also let $C$ denote the lower right-hand corner of $M(2)$. \ Since we
want $\operatorname{rank}\;M(2)=\operatorname{rank}\;M(1)=3$, we must define $%
C:=B^{T}M(1)^{-1}B$. \ Again, we need to verify that $M(2)_{46}=M(2)_{55}$,
i.e., $C_{13}=C_{22}$. \ A straightforward calculation shows that 
\begin{equation*}
C_{13}=bcdr.
\end{equation*}%
When $c>e$, we have $r=\frac{ef}{d}$, and another calculation shows that 
\begin{equation*}
C_{22}=\frac{b^{2}ce^{2}}{a};
\end{equation*}%
it is then immediate that $C_{13}=C_{22}$. \ When $c=e$, we have $r=c$, and
in this case $C_{13}=C_{22}=bc^{2}d$, as desired.

The proof of the Theorem is now complete.
\end{proof}

\section{\label{concrete}Description of the Representing Measure}

In this section we provide a concrete description of the Berger measure for
the subnormal completion in Theorem \ref{quartic}. \ We have already
observed that when $\operatorname{rank}\;M(1)=1$, the representing measure is $\mu
=\delta _{(a,b)}$. \ When $\operatorname{rank}\;M(1)=2$ (and the columns $1$ and $X$
linearly independent), there is a $2$-atomic representing measure, with
atoms $(0,y_{0})$ and $(c,y_{c})$ given by (\ref{y0}) and (\ref{yc}); thus, $%
\mu =\rho _{(0,y_{0})}\delta _{(0,y_{0})}+\rho _{(c,y_{c})}\delta
_{(c,y_{c})}$. \ To find the densities $\rho _{(0,y_{0})}$ and $\rho
_{(c,y_{c})}$, we use the first two moments: $\int d\mu =\rho
_{(0,y_{0})}+\rho _{(c,y_{c})}=1$ and $\int s\;d\mu =c\rho _{(c,y_{c})}=a$.
\ It follows that the densities are $\rho _{(0,y_{0})}=1-\frac{a}{c}$ and $%
\rho _{(c,y_{c})}=\frac{a}{c}$. \ Thus, $\mu =(1-\frac{a}{c})\delta
_{(0,y_{0})}+\frac{a}{c}\delta _{(c,y_{c})}$.

We now focus on the case $\operatorname{rank}\;M(1)=3$. \ Since $M(1)$ is
invertible, the last three columns of the flat extension $M(2)$ can be
written in terms of the first three columns; that is, the columns labeled $%
X^{2}$, $YX$ and $Y^{2}$ are linear combinations of $1$, $X$ and $Y$. \ Each
of these column relations is associated with a quadratic polynomial in $x$
and $y$, whose zero sets give rise to the so-called \textit{algebraic variety%
} of $\hat{\Omega}_{3}$ \cite{tcmp7}; concretely, $\mathcal{V}(\hat{\Omega}%
_{3}):=\bigcap_{p(X,Y)=0\text{, }\deg \;p\leq 2}\mathcal{Z}(p)$, where $%
\mathcal{Z}(p)$ denotes the zero set of $p$. \ In our case, the three column
relations are 
\begin{eqnarray*}
X^{2} &=&cX \\
YX &=&fX \\
Y^{2} &=&\frac{be(f-d)}{a(c-e)}X+\frac{cd-ef}{c-e}Y.
\end{eqnarray*}%
The associated zero sets are 
\begin{eqnarray*}
\{(x,y) &:&x=0\text{ or }x=c\} \\
\{(x,y) &:&x=0\text{ or }y=f\} \\
\{(x,y) &:&y^{2}=\frac{be(f-d)}{a(c-e)}x+\frac{cd-ef}{c-e}y\}.
\end{eqnarray*}%
Let $z:=\frac{cd-ef}{c-e}$ and observe that $z>0$ by (\ref{cdef}). $\ $The
algebraic variety of $\hat{\Omega}_{3}$ is then $\mathcal{V}(\hat{\Omega}%
_{3})=\{(0,0),(0,z),(c,f)\}$ and these are the three atoms of the unique
representing measure for $M(2)$. \ To find the densities, we use the first
three moments, $\gamma _{00}$, $\gamma _{01}$ and $\gamma _{10}$: 
\begin{equation*}
\left\{ 
\begin{array}{ccc}
\rho _{(0,0)}+\rho _{(0,z)}+\rho _{(c,f)} & = & 1 \\ 
\rho _{(c,f)}c & = & a \\ 
\rho _{(0,z)}z+\rho _{(c,f)}f & = & b.%
\end{array}%
\right.
\end{equation*}%
We obtain 
\begin{eqnarray*}
\rho _{(c,f)} &=&\frac{a}{c} \\
\rho _{(0,z)} &=&\frac{1}{z}(b-\frac{a}{c}f)=\frac{b(c-e)^{2}}{c(cd-ef)} \\
\rho _{(0,0)} &=&1-\rho _{(0,z)}-\rho _{(c,f)}=\frac{1}{ab(cd-ef)}\det
\;M(1).
\end{eqnarray*}%
Thus, the representing measure is 
\begin{equation}
\mu =\frac{\det \;M(1)}{ab(cd-ef)}\delta _{(0,0)}+\frac{b(c-e)^{2}}{c(cd-ef)}%
\delta _{(0,z)}+\frac{a}{c}\delta _{(c,f)}.  \label{measure}
\end{equation}%
Direct calculation shows that $\int s^{2}\;d\mu (s,t)=\frac{a}{c}c^{2}=ac$, $%
\int st\;d\mu (s,t)=\frac{a}{c}cf=af=be$, and 
\begin{eqnarray*}
\int t^{2}\;d\mu (s,t) &=&\frac{b(c-e)^{2}}{c(cd-ef)}z^{2}+\frac{a}{c}f^{2}=%
\frac{b(c-e)^{2}}{c(cd-ef)}(\frac{cd-ef}{c-e})^{2}+\frac{bef}{c} \\
&=&\frac{b(cd-ef)}{c}+\frac{bef}{c}=bd,
\end{eqnarray*}%
so that $\mu $ correctly interpolates $\Omega _{1}$. \ 

Recall now that the marginal measures $\nu ^{X}$ and $\nu ^{Y}$ associated
to a Borel measure $\nu $ on the Cartesian product $X\times Y$ are given by $%
\nu ^{X}(E):=\nu (E\times Y)$ and $\nu ^{Y}(F):=\nu (X\times F)$, for $E$
and $F$ Borel sets. \ In the specific case of the measure $\mu $ in (\ref%
{measure}), observe that the marginal measures $\mu ^{X}$ and $\mu ^{Y}$ are 
$(1-\frac{a}{c})\delta _{0}+\frac{a}{c}\delta _{c}$ and $\frac{\det \;M(1)}{%
ab(cd-ef)}\delta _{0}+\frac{b(c-e)^{2}}{c(cd-ef)}\delta _{z}+\frac{a}{c}%
\delta _{f}$, respectively. \ While $\mu ^{X}$ is always $2$-atomic, $\mu
^{Y}$ is $3$-atomic if and only if $z\neq f$. \ When $\mu ^{Y}$ is $3$%
-atomic, its moments (which are also the moments of an associated unilateral
weighted shift $W_{\eta }$) satisfy the recursive relation $\gamma
_{n+2}=-fz\gamma _{n}+(f+z)\gamma _{n+1}$ (all $n\geq 1)$, with $\gamma
_{0}=1$, $\gamma _{1}=b$ and $\gamma _{2}=bd$. \ It is easy to see that the
restriction of $W_{\eta }$ to the invariant subspace $M_{1}$ defined in (\ref%
{mh}) has Berger measure $\frac{1}{\gamma _{1}}t\;d\mu ^{Y}(t)=\frac{\rho
_{(0,z)}z}{b}\delta _{z}(t)+\frac{\rho _{(c,f)}f}{b}\delta _{f}(t)$, whose
recursive coefficients are $-zf$ and $z+f$, respectively.

On the other hand, it is indeed possible to have $z=f$, which occurs
precisely when $d=f$. \ In that case, the three atoms of $\mu $ are $(0,0)$, 
$(0,f)$ and $(c,f)$, and the unilateral weighted shift associated with $\mu
^{Y}$ is $W_{(b,d,d,\cdots )}$. The reader will note that the location of
these atoms can also be predicted by Theorem \ref{main}, once we observe
that $\operatorname{rank}\;M_{x}(2)=1$ and $\operatorname{rank}\;M_{y}(2)=2$.

\section{\label{flat}Flat Extensions May Not Exist}

We now present an example of a set $\Omega _{3}$ for which the associated
moment matrix $M(2)$ admits a representing measure, but such that $M(2)$ has no 
flat extension $M(3)$. \ Thus, while Theorem \ref{main} provides a general sufficient
condition for solving SCP, not all SCP will fit that framework, and their
associated moment matrices $M(\Omega _{\tilde{m}})$ will require a sequence
of moment matrix extensions $M(\Omega _{\tilde{m}+2})$, $\cdots $ , $%
M(\Omega _{\tilde{m}+2k})$, with $M(\Omega _{\tilde{m}+2k})$ admitting a
flat extension $M(\Omega _{\tilde{m}+2(k+1)})$.

The example is motivated by the construction in \cite[Examples 1.13 and 5.6]%
{tcmp6}, and also by \cite[Proposition 1.12]{tcmp6}, which states that a TMP
and its image under a degree-one transformation of the base space are
equivalent as moment problems. \ In particular, the qualitative aspects of
TMP are preserved under degree-one transformations; our idea is therefore to
``translate'' \cite[Example 1.13]{tcmp6} three units to the right and four
units up, so that the support of the $6$-atomic representing measure in %
\cite[Example 1.13]{tcmp6} will land in the positive quadrant. \ (We note
that to produce a valid representing measure for SCP, it suffices to have
all atoms in the nonnegative quadrant, and at least one atom in the positive
quadrant.) \ To effectuate the above mentioned translation, we recall the
definition of the Riesz functional $L_{\gamma }$ associated to a TMP. \ The
linear functional $L_{\gamma }$ acts on polynomials by $L(y^{i}x^{j}):=%
\gamma _{ij}$. \ Given the moments $\gamma _{ij}$, one can translate the TMP
by $h$ units in the horizontal direction and $k$ units in the vertical
direction by letting $\tilde{\gamma}_{ij}\equiv $ $\tilde{L}%
(v^{i}u^{j}):=L_{\gamma }((v+4)^{i}(u+3)^{j})$. \ The associated moments of
degree $4$ are:%
\begin{equation*}
\begin{array}{ccccccccccc}
\gamma _{00}=1 &  &  &  &  &  & \tilde{\gamma}_{00}=1 &  &  &  &  \\ 
\gamma _{01}=1 & \gamma _{10}=1 &  &  &  &  & \tilde{\gamma}_{01}=4 & \tilde{%
\gamma}_{10}=5 &  &  &  \\ 
\gamma _{02}=2 & \gamma _{11}=0 & \gamma _{20}=3 &  &  &  & \tilde{\gamma}%
_{02}=17 & \tilde{\gamma}_{11}=19 & \tilde{\gamma}_{20}=27 &  &  \\ 
\gamma _{03}=4 & \gamma _{12}=0 & \gamma _{21}=0 & \gamma _{30}=9 &  &  & 
\tilde{\gamma}_{03}=76 & \tilde{\gamma}_{12}=77 & \tilde{\gamma}_{21}=97 & 
\tilde{\gamma}_{30}=157 &  \\ 
\gamma _{04}=9 & \gamma _{13}=0 & \gamma _{22}=0 & \gamma _{31}=0 & \gamma
_{40}=28 &  & \tilde{\gamma}_{04}=354 & \tilde{\gamma}_{13}=331 & \tilde{%
\gamma}_{22}=371 & \tilde{\gamma}_{31}=535 & \tilde{\gamma}_{40}=972.%
\end{array}%
\end{equation*}%
For example, 
\begin{eqnarray*}
\tilde{\gamma}_{21} &=&L_{\gamma }((v+4)^{2}(u+3))=L_{\gamma
}((v^{2}+8v+16)(u+3)) \\
&=&L_{\gamma }(v^{2}u+8vu+3v^{2}+16u+24v+48)=\gamma _{21}+8\gamma
_{11}+3\gamma _{20}+16\gamma _{01}+24\gamma _{10}+48\gamma _{00} \\
&=&0+8\cdot 0+3\cdot 3+16\cdot 1+24\cdot 1+48\cdot 1=97.
\end{eqnarray*}

With the new moments at hand, we form the matrix $M(2)$. \ The corresponding
weights are: 
\begin{equation}
\begin{array}[t]{llll}
\alpha _{03}=\frac{\sqrt{535}}{\sqrt{157}} &  &  &  \\ 
\alpha _{02}=\frac{\sqrt{97}}{3\sqrt{3}} & \alpha _{12}=\frac{\sqrt{371}}{%
\sqrt{97}} &  &  \\ 
\alpha _{01}=\frac{\sqrt{19}}{\sqrt{5}} & \alpha _{11}=\frac{\sqrt{77}}{%
\sqrt{19}} & \alpha _{21}=\frac{\sqrt{331}}{\sqrt{77}} &  \\ 
\alpha _{00}=2 & \alpha _{10}=\frac{\sqrt{17}}{2} & \alpha _{20}=\frac{2%
\sqrt{19}}{\sqrt{17}} & \alpha _{30}=\frac{\sqrt{17}}{\sqrt{38}}%
\end{array}
\label{omega3}
\end{equation}%
\begin{equation*}
\begin{array}[t]{llll}
\beta _{03}=\frac{18\sqrt{3}}{\sqrt{157}} &  &  &  \\ 
\beta _{02}=\frac{\sqrt{157}}{3\sqrt{3}} & \beta _{12}=\frac{\sqrt{535}}{%
\sqrt{97}} &  &  \\ 
\beta _{01}=\frac{3\sqrt{3}}{\sqrt{5}} & \beta _{11}=\frac{\sqrt{97}}{\sqrt{%
19}} & \beta _{21}=\frac{\sqrt{371}}{\sqrt{97}} &  \\ 
\beta _{00}=\sqrt{5} & \beta _{10}=\frac{\sqrt{19}}{2} & \beta _{20}=\frac{%
\sqrt{77}}{\sqrt{17}} & \beta _{30}=\frac{\sqrt{331}}{2\sqrt{19}}.%
\end{array}%
\end{equation*}

\begin{example}
\label{noflat}Let $\Omega _{3}$ be given by (\ref{omega3}) and let $%
M(2)\equiv M(2)(\Omega _{3})$ the its associated moment matrix, with entries
built from the data $\tilde{\gamma}_{ij}$. \ Let $M(3)$ be a positive
semi-definite, recursively generated, moment matrix extension of $M(2)$. \
Then $\operatorname{rank}\;M(3)>\operatorname{rank}\;M(2)$. \ As a consequence, $M(2)$
admits no flat extension $M(3)$. \ For, consider a moment matrix extension 
\begin{equation*}
M(3):=\left( 
\begin{array}{cccccccccc}
1 & 4 & 5 & 17 & 19 & 27 & 76 & 77 & 97 & 157 \\ 
4 & 17 & 19 & 76 & 77 & 97 & 354 & 331 & 371 & 535 \\ 
5 & 19 & 27 & 77 & 97 & 157 & 331 & 371 & 535 & 972 \\ 
17 & 76 & 77 & 354 & 331 & 371 & \tilde{\gamma}_{05} & \tilde{\gamma}_{14} & 
\tilde{\gamma}_{23} & \tilde{\gamma}_{32} \\ 
19 & 77 & 97 & 331 & 371 & 535 & \tilde{\gamma}_{14} & \tilde{\gamma}_{23} & 
\tilde{\gamma}_{32} & \tilde{\gamma}_{41} \\ 
27 & 97 & 157 & 371 & 535 & 972 & \tilde{\gamma}_{23} & \tilde{\gamma}_{32}
& \tilde{\gamma}_{41} & \tilde{\gamma}_{50} \\ 
76 & 354 & 331 & \tilde{\gamma}_{05} & \tilde{\gamma}_{14} & \tilde{\gamma}%
_{23} & \tilde{\gamma}_{06} & \tilde{\gamma}_{15} & \tilde{\gamma}_{24} & 
\tilde{\gamma}_{33} \\ 
77 & 331 & 371 & \tilde{\gamma}_{14} & \tilde{\gamma}_{23} & \tilde{\gamma}%
_{32} & \tilde{\gamma}_{15} & \tilde{\gamma}_{24} & \tilde{\gamma}_{33} & 
\tilde{\gamma}_{42} \\ 
97 & 371 & 535 & \tilde{\gamma}_{23} & \tilde{\gamma}_{32} & \tilde{\gamma}%
_{41} & \tilde{\gamma}_{24} & \tilde{\gamma}_{33} & \tilde{\gamma}_{42} & 
\tilde{\gamma}_{51} \\ 
157 & 535 & 972 & \tilde{\gamma}_{32} & \tilde{\gamma}_{41} & \tilde{\gamma}%
_{50} & \tilde{\gamma}_{33} & \tilde{\gamma}_{42} & \tilde{\gamma}_{51} & 
\tilde{\gamma}_{06}%
\end{array}%
\right) ,
\end{equation*}%
where the moments of degree $5$ and $6$ are new. \ A direct computation
shows that $\operatorname{rank}\;M(2)=5$, and that $(X-3)(Y-4)=0$, that is, $%
YX=4X+3Y-12$. \ In any positive semi-definite, recursively generated,
extension $M(3)$ this column relation would still be valid, and it would
also give rise to two new column relations, namely $YX^{2}=4X^{2}+3YX-12X$
and $Y^{2}X=4YX+3Y^{2}-12Y$. \ These three identities lead at once to the
values $\tilde{\gamma}_{14}=1497$, $\tilde{\gamma}_{23}=1513$, $\tilde{\gamma%
}_{32}=1925$, $\tilde{\gamma}_{41}=3172$, $\tilde{\gamma}_{15}=243+4\tilde{%
\gamma}_{05}$, $\tilde{\gamma}_{24}=6555$, $\tilde{\gamma}_{33}=7375$, $%
\tilde{\gamma}_{42}=10796$, and $\tilde{\gamma}_{51}=1024+3\tilde{\gamma}%
_{50}$. \ Now, since the compression of $M(2)$ to the rows and columns
indexed by $1$, $X$, $Y$, $X^{2}$ and $Y^{2}$ is invertible, we can find
coefficients $A_{1}$, $A_{X}$, $A_{Y}$, $A_{X^{2}}$ and $A_{Y^{2}}$ such
that 
\begin{equation}
A_{1}[1]_{\mathcal{B}}+A_{X}[X]_{\mathcal{B}}+A_{Y}[Y]_{\mathcal{B}%
}+A_{X^{2}}[X^{2}]_{\mathcal{B}}+A_{Y^{2}}[Y^{2}]_{\mathcal{B}}=[X^{3}]_{%
\mathcal{B}}\text{,}  \label{relation}
\end{equation}%
where $[\cdot ]_{\mathcal{B}}$ denotes the compression of a column to $%
\mathcal{B}:=\{1,X,Y,X^{2},Y^{2}\}$. \ A calculation using \textit{%
Mathematica} \cite{Wol} reveals that $A_{1}=-25513+15\tilde{\gamma}_{05}$, $%
A_{X}=13587-8\tilde{\gamma}_{05}$, $A_{Y}=1$, $A_{X^{2}}=-1692+\tilde{\gamma}%
_{05}$ and $A_{Y^{2}}=0$. \ If $M(3)$ were a flat extension of $M(2)$, an
identity similar to (\ref{relation}) should hold for the last row in $M(3)$,
that is, 
\begin{equation}
A_{1}[1]_{\{X^{3}\}}+A_{X}[X]_{\{X^{3}\}}+A_{Y}[Y]_{\{X^{3}%
\}}+A_{X^{2}}[X^{2}]_{\{X^{3}\}}+A_{Y^{2}}[Y^{2}]_{\{X^{3}\}}=[X^{3}]_{%
\{X^{3}\}}.  \label{relation2}
\end{equation}%
Using \textit{Mathematica} again, it is easy to check that $%
A_{1}[1]_{\{X^{3}\}}+A_{X}[X]_{\{X^{3}\}}+A_{Y}[Y]_{\{X^{3}%
\}}+A_{X^{2}}[X^{2}]_{\{X^{3}\}}+A_{Y^{2}}[Y^{2}]_{\{X^{3}\}}=7376$, while $%
[X^{3}]_{\{X^{3}\}}=\tilde{\gamma}_{33}=7375$. \ It follows that $M(3)$
cannot be a flat extension of $M(2)$.
\end{example}

\begin{remark}
The SCP in Example \ref{noflat} does admit a solution, and the subnormal
completion has a $6$-atomic Berger measure. \ We see this after we observe
that the positive semi-definite moment matrix extension $M(3)$, while not a
flat extension of $M(2)$, does admit a flat extension $M(4)$. \ Rather than
showing the details here, we refer the reader to \cite[Proposition 5.5 and
Example 5.6]{tcmp6}; the representing measure constructed there must be
translated three units to the right and four units up to give rise to the
Berger measure that solves SCP in Example \ref{noflat}. \ 
\end{remark}

\textit{Acknowledgments.} \ The authors are deeply indebted to the referee
for a number of comments and observations that helped improve the content
and presentation of this paper. \ The examples, and portions of the proof of
Theorem \ref{quartic} were obtained using calculations with the software
tool \textit{Mathematica \cite{Wol}}.

\end{document}